\input ./style/arxiv-vmsta.cfg
\documentclass[numbers,compress,v1.0.1]{vmsta}
\usepackage{vtexurl}

\volume{3}
\issue{2}
\pubyear{2016}
\firstpage{119}
\lastpage{131}
\doi{10.15559/16-VMSTA55}

\startlocaldefs

\allowdisplaybreaks

\newtheorem*{theorem*}{Theorem}
\newtheorem{theorem}{Theorem}[section]

\newtheorem{corollary}{Corollary}[section]
\theoremstyle{remark}

\theoremstyle{definition}

\endlocaldefs

\begin{document}
\begin{frontmatter}

\title{On fractal faithfulness and fine fractal properties of~random
variables with independent $\boldsymbol{Q^*}$-digits}

\author[a]{\inits{M.}\fnm{Muslem}\snm{Ibragim}}
\email{muslemhussen1978@yahoo.com}
\address[a]{National Pedagogical Dragomanov University, \\Pyrogova str.
9, 01030, Kyiv, Ukraine}

\author[a,b]{\inits{G.}\fnm{Grygoriy}\snm{Torbin}\corref{cor1}}
\cortext[cor1]{Corresponding author.}\email{torbin7@gmail.com}
\address[b]{Institute for Mathematics of National Academy
of Sciences, \\Tereshchenkivska str. 3, 01601, Kyiv, Ukraine}


%
%
%

\markboth{M. Ibragim, G. Torbin}{On fractal faithfulness and fine
fractal properties of random variables with independent $Q^*$-digits}

\begin{abstract}
We develop a new technique to prove the faithfulness of the
Hausdorff--Besicovitch dimension calculation of the family $\varPhi(Q^*)$
of cylinders generated by $Q^*$-expansion of real numbers. All known
sufficient conditions for the family $\varPhi(Q^*)$ to be faithful for the
Hausdorff--Besicovitch dimension calculation use different restrictions
on entries $q_{0k}$ and $q_{(s-1)k}$. We show that these restrictions
are of purely technical nature and can be removed.
Based on these new results, we study fine fractal properties of random
variables with independent $Q^*$-digits.
\end{abstract}



%
\received{20 May 2016}
%
\accepted{3 June 2016}
\publishedonline{9 June 2016}

\end{frontmatter}

\section{Introduction}

Hausdorff measures and the Hausdorff dimension are important tools in
the study of fractals and singularly continuous probability measures.
The determination or even estimation of the Hausdorff dimension of a
set or measure is the crucial problem in fractal analysis, and a lot of
research papers were devoted to these problems. Because of this reason,
many interesting methods for the simplification of the procedure of the
determination of the Hausdorff dimension were invented and developed
during the last 20 years. One approach to such a simplification
consists in some restrictions of admissible coverings. This idea came
from Besicovitch's works and has been used by Rogers and Taylor to
construct comparable net measures \cite{Rogers} as approximations of
the Hausdorff measures. In this paper, we develop this approach via
construction of net coverings that lead to a special family of net
measures, which are more general that comparable ones. We discuss the
notion of faithfulness and nonfaithfulness of the family of cylinders
generated by different systems of numerations for the Hausdorff
dimension calculation.

Let us shortly recall that the \textit{$\alpha$-dimensional Hausdorff measure}
of a set $E\subset[0,1]$ with respect to a given fine family of
coverings $ \varPhi$ is defined by
\[
H^{\alpha} (E, \varPhi)= \lim\limits_{\epsilon\to0 } ~~ \inf\limits
_{|E_{j} |\le\epsilon}~ \sum_{j} |E_{j}|^{\alpha}
= \lim\limits_{\epsilon\to0 } H_{\epsilon}^{\alpha} (E,\, \varPhi),
\]
where the infimum is taken over all at most countable $\epsilon
$-coverings $\{ E_{j} \} $ of $E$, $E_{j} \in\varPhi$.
The nonnegative number
\[
\dim_{H} (E,\, \varPhi)=\inf\big\{ \alpha:\, \, H^{\alpha} (E,\, \varPhi)=0\big\}
\]
\noindent is called the Hausdorff dimension of the set $E\subset[0,1]$
w.r.t.\ the family $\varPhi$.
If $\varPhi$ is the family of all subsets of $[0, 1]$ or $\varPhi$ coincides
with the family of
all closed (open) subintervals of [0,1], then $\dim_{H} (E,\, \varPhi)$
is equal to the classical Hausdorff dimension $\dim_{H} (E)$ of a
subset $E \subset[0,1]$.

A fine covering family $\varPhi$ is said to be a \textit{faithful family
of coverings}~(\textit{nonfaithful family of coverings}) for the
Hausdorff dimension calculation on $[0,1]$ if
\begin{align*}
&\quad\dim_{H} (E,\varPhi)=\dim_{H} (E), ~~~\forall E\subseteq[0,1]\\
&(\mbox{resp.} ~~\exists E\subseteq[0,1]: \dim_{H} (E,\varPhi)\neq\dim
_{H} (E)).
\end{align*}

It is clear that any family $\varPhi$ of comparable net-coverings (\xch{i.e.}{i.,e.},
net-coverings that generate comparable net-measures) is faithful.
Conditions for Vitali coverings to be faithful were studied by many
authors (see, e.g., \cite{AlbeverioTorbin2005, AILT, AKNT1 and 3,
Billingsley61} and \xch{the references}{references} therein). First steps in this direction
have been done by Besicovitch~\cite{Besicovitch}, who proved the
faithfulness of the family of cylinders of a binary expansion. His
result was extended by~Billingsley \cite{Billingsley61} to the family
of $s$-adic cylinders, \xch{by Turbin and Pratsiovytyi \cite{TuP}}{by~Pratsiovytyi \cite{TuP}} to the family of
$Q$-$S$-cylinders, and by Albeverio and~Torbin \cite
{AlbeverioTorbin2005} to the family of $Q^*$-cylinders for the matrices
$Q^*$ with elements $p_{0k}, p_{(s-1)k}$ bounded away from zero.

In all these papers, their authors used essentially the same approach
to prove the faithfulness of the corresponding family of coverings: it
has been shown that there exist positive constants $C$ and $n_0 \in N$
such that, for any $\varepsilon>0$ and for any interval $(a; b)$ with
$b-a < \varepsilon$, there exist at most $n_0$ cylinders from fine
covering families that cover the interval $(a,b)$ and their lengths do
not exceed the value $C(b-a)$. It is rather obvious that such families
$\varPhi$ of cylinders generates comparable Hausdorff measures \cite
{Rogers}, and, therefore, they are faithful for the Hausdorff dimension
calculation. Albeverio et al.\  \cite{AKNT1 and 3} correctly mentioned
that it was rather paradoxical that initial examples of nonfaithful
families of coverings first appeared in the two-dimensional case (as a
result of active studies of self-affine sets during the last decade of
XX \xch{century, see,}{century (see,} e.g.,~\cite{BB}). The family of cylinders of the
classical continued fraction expansion can probably be considered as
the first (and rather unexpected) example of nonfaithful
one-dimensional net-family of coverings \cite{PerTor}. By using approach which has been invented by Yuval Peres to prove the nonfaithfulness of the family of continued fraction cylinders, in \cite{AKNT1 and 3}
the nonfaithfulness of the family $\varPhi(Q_\infty)$ of cylinders of the
$Q_\infty$-ex\-pan\-sion with polynomially decreasing elements~$\{q_i\}$
has been proven. This shows, in particular, that the family of
cylinders of the classical L\"uroth expansion is nonfaithful. Rather
general sufficient conditions for $\varPhi(Q_\infty)$ to be faithful were
also obtained in \cite{AKNT1 and 3, NT_TVIMS12}.

In 2012, Ibragim and Torbin \cite{IbragimTorbin2} developed a new
method to prove the faithfulness of the family of cylinders of
$Q^*$-expansion for the matrices $Q^*$ with elements $p_{0k},
p_{(s-1)k}$ not tending to zero ``too quickly.'' In particular, they
proved the following result.\smallskip

\begin{theorem*}
Let $q_k^*:=\max\{q_{0k},q_{1k},\ldots,q_{s-1k}\}$. If
\begin{equation}\label{teorema 2 o ddoveritelnosti}
\left\{
\begin{array}{l}
\lim\limits_{k\to\infty}\frac{\ln{q_{0,k}}}{\ln
(q_1^*q_2^*\ldots q_k^*)}=0,\\[6pt]
\lim\limits_{k\to\infty}\frac{\ln{q_{s-1,k}}}{\ln(q_1^*q_2^*\ldots q_k^*)}=0,
\end{array}
\right.
\end{equation}
then\vspace{-3pt}
\[
\dim_H(E)=\dim_H\bigl(E,\varPhi(Q^*)\bigr),\quad \forall E\subset[0,1].
\]
\end{theorem*}

This theorem, a generalization of \cite{AlbeverioTorbin2005}, extended
the family of faithful coverings generated by cylinders of
$Q^*$-expansion. In particular, we can easily apply this theorem to
prove the faithfulness of the family of cylinders generated by the
matrix\vspace{-3pt}
\[
Q^*=
\begin{pmatrix}

\frac1{10}&\ldots&\frac1{10k}&\ldots\\[3pt]

{{\frac12}-{\frac1{10}}}&\ldots&\
{{\frac12}-{\frac1{10k}}}&\ldots\\[3pt]

{{\frac12}-{\frac1{10}}}&\ldots&\
{{\frac12}-{\frac1{10k}}}&\ldots\\[3pt]

\frac1{10}&\ldots&\frac1{10k}&\ldots
\end{pmatrix}
.
\]

On the other hand, if $p_{0k}, p_{(s-1)k}$ tend to zero ``too quickly''
(e.g., $s=4$, $q_{0k}=q_{3k}=\frac{1}{10^k}$, $q_{1k}=q_{2k}=\frac
{1}{2} - \frac{1}{10^k}$), then the above theorem does not work.

In the next section, we develop a new approach to prove the
faithfulness of families of coverings and prove essentially new
sufficient conditions for $Q^*$-cylinders to be faithful (we do not
need any information about the boundedness from zero of the elements
$q_{0k}$ and $q_{(s-1)k}$ or any information about the rate of their
convergence to zero).

\section{On new sufficient conditions of fractal faithfulness for the
family of cylinders of $Q^*$-expansions}

\begin{theorem}
\label{Th2.1}
Let $q_k:= \max_{i} q_{ik}$, let
\[
S(m, \delta) := \sum\limits_{k=1}^{\infty} \left( \prod\limits
_{i=m+1}^{m+k}q_i\right)^{\delta},
\]
and let
\[
S(\delta) := \sup\limits_m S(m, \delta).
\]

If
\begin{equation}\label{sufficient cond for Q* faithfulness}
S(\delta)< +\infty, \quad \forall\delta>0,
\end{equation}
then the family $\varPhi(Q^*)$ of cylinders generated by $Q^*$-expansion
of real numbers is faithful for the calculation of the Hausdorff
dimension on the unit interval, that is,
\[
\dim_H E = \dim_H \bigl(E, \varPhi(Q^*)\bigr), \quad \forall E \subset[0,1].
\]
\end{theorem}

\begin{proof} It is clear that for the determination of the Hausdorff
dimension of subsets from $[0,1]$ it suffices to consider coverings by
intervals $(a_j, b_j)$, where $a_j$ and $b_j$ belong to a set $A$ that
is dense in $[0,1]$. Let $A$ be the set of all $Q^*$-irrational points,
that is, the set of points that are not end-points of $Q^*$-cylinders
(the $Q^*$-expansion of these points does not contain digits 0 or $s-1$
in a period).

Let $E$ be an arbitrary subset of $[0, 1]$. Let us fix $\varepsilon>0$
and $\alpha>0$. Let $\{E_j\}$ be an arbitrary $\varepsilon$-covering of
the set $E$, $E_j=(a_j, b_j)$, $a_j\in A$, $b_j\in A$,
$|E_j|<\varepsilon$.

For the interval $E_j$, there exists a unique cylinder $\varDelta_{\alpha
_1\alpha_2\ldots\alpha_{n_j}}$ containing~$E_j$ such that any cylinder
of a higher rank does not contain $E_j$. In the case where $a_j$ and
$b_j$ belong to different cylinders of the first rank, we define $\varDelta
_{\alpha_1\alpha_2\ldots\alpha_{n_j}} :=[0, 1]$.

Let us split $\varDelta_{\alpha_1\ldots\alpha_{n_j}}$ on the next rank
cylinders. From the maximality of the rank of the cylinder $\varDelta
_{\alpha_1\alpha_2\ldots\alpha_{n_j}}$ it follows that there exists at
least one point that is an end-point of a cylinder of rank $n_j+1$ and
belongs to the interval $(a_j, b_j)$. It is clear that the point
\[
c_j=\varDelta_{\alpha_1(a_j)\ldots\alpha_{n_j}(a_j)(\alpha
_{n_j+1}(a_j)+1)0\ldots0\ldots}
\]
possesses such properties.

Let $M_0 = M_0(j)$ be a family of cylinders of rank $n_j+1$ belonging
to $(a_j, b_j)$. It is clear that $M_0$ contains less than $s$
cylinders (if the points $a_j$ and~$b_j$ belong to neighboring
cylinders of rank $n_j+1$, then $M_0$ is empty). Therefore, the $\alpha
$-volume of these cylinders does not exceed $s |E_j|^{\alpha}$.

Let $d_j:= \sup M_0 = \varDelta_{\alpha_1(a_j)\ldots\alpha_{n_j}(a_j)
\alpha_{n_j+1}(b_j)~ 0\ldots0\ldots}$.

To cover the set $E_j$ by cylinders from $\varPhi(Q^*)$, let us cover the
sets $(a_j, c_j)$ and $[d_j, b_j)$ separately.
Let us choose $\delta\in(0, \alpha)$.

First, let us estimate the $\alpha$-volume of coverings of the set
$(a_j, c_j)$.

Let $L_1=L_1(j)$ be the family of all cylinders of rank $n_j+2$
belonging to the cylinder $\varDelta_{\alpha_1(a_j)\alpha_2(a_j)\ldots
\alpha_{n_j+1}(a_j)}$ and to the set $(a_j, c_j]$.
Let
\[
A_1=A_1(j):= \bigl\{i: i\in\bigl\{ \alpha_{n_j+2}(a_j)+1,\dots, s-1 \bigr\} \bigr\}.
\]

The corresponding $\alpha$-volume of these cylinders is equal to
\begin{align*}
&\sum\limits_{i \in A_1} |\varDelta_{\alpha_1(a_j)\alpha_2(a_j)\ldots\alpha
_{n_j+1}(a_j) ~ i}|^{\alpha} \leq s \cdot\max\limits_{i \in A_1}
|\varDelta_{\alpha_1(a_j)\alpha_2(a_j)\ldots \alpha_{n_j+1}(a_j)~
i}|^{\alpha}\\
&\quad= s \cdot\max\limits_{i \in A_1} \bigl(|\varDelta_{\alpha_1(a_j)\alpha
_2(a_j)\ldots \alpha_{n_j+1}(a_j)~ i}|^{\alpha-\delta} |\varDelta_{\alpha
_1(a_j)\alpha_2(a_j)\ldots\alpha_{n_j+1}(a_j)~ i}|^{\delta}\bigr)\\
&\quad\leq s |E_j|^{\alpha-\delta} \cdot\max\limits_{i \in A_1} |\varDelta
_{\alpha_1(a_j)\alpha_2(a_j)\ldots\alpha_{n_j+1}(a_j)~ i}|^{\delta}\\
&\quad\leq s |E_j|^{\alpha-\delta} \cdot\bigl(q_{n_j+2} ~ |\varDelta_{\alpha
_1(a_j)\alpha_2(a_j)\ldots\alpha_{n_j}(a_j)}|\bigr)^{\delta} \leq s |E_j|
^{\alpha-\delta} \cdot q_{n_j+2}^{\delta}.
\end{align*}

Let $L_2=L_2(j)$ be the family of all cylinders of rank $n_j+3$
belonging to the cylinder $\varDelta_{\alpha_1(a_j)\alpha_2(a_j)\ldots
\alpha_{n_j+2}(a_j)}$ and to the set $(a_j, c_j]$. Let
\[
A_2=A_2(j):= \bigl\{i: i\in\bigl\{ \alpha_{n_j+3}(a_j)+1, \dots, s-1 \bigr\} \bigr\}.
\]

The corresponding $\alpha$-volume of these cylinders is equal to
\begin{align*}
&\sum\limits_{i \in A_2} |\varDelta_{\alpha_1(a_j)\alpha_2(a_j)\ldots\alpha
_{n_j+2}(a_j) ~ i}|^{\alpha}\\
&\quad \leq s \cdot\max\limits_{i \in A_2}
|\varDelta_{\alpha_1(a_j)\alpha_2(a_j)\ldots \alpha_{n_j+2}(a_j)~
i}|^{\alpha}\\
&\quad= s \cdot\max\limits_{i \in A_2} \bigl(|\varDelta_{\alpha_1(a_j)\alpha
_2(a_j)\ldots \alpha_{n_j+2}(a_j)~ i}|^{\alpha-\delta} |\varDelta_{\alpha
_1(a_j)\alpha_2(a_j)\ldots\alpha_{n_j+2}(a_j)~ i}|^{\delta}\bigr)\\
&\quad\leq s |E_j|^{\alpha-\delta} \cdot\max\limits_{i \in A_2} |\varDelta
_{\alpha_1(a_j)\alpha_2(a_j)\ldots\alpha_{n_j+2}(a_j)~ i}|^{\delta}\\
&\quad\leq s |E_j|^{\alpha-\delta} \cdot\bigl(q_{n_j+2} q_{n_j+3} ~ |\varDelta
_{\alpha_1(a_j)\alpha_2(a_j)\ldots\alpha_{n_j}(a_j)}|\bigr)^{\delta} \leq s
\cdot|E_j| ^{\alpha-\delta} \cdot(q_{n_j+2} q_{n_j+3})^\delta.
\end{align*}

Similarly, let $L_k=L_k(j)$ be the family of all cylinders of rank
$n_j+k+1$ belonging to the cylinder $\varDelta_{\alpha_1(a_j)\alpha
_2(a_j)\ldots\alpha_{n_j+k}(a_j)}$ and to the set $(a_j, c_j]$.
Let
\[
A_k=A_k(j):= \bigl\{i: i\in\bigl\{ \alpha_{n_j+k+1}(a_j)+1, \dots, s-1 \bigr\} \bigr\}.
\]

The corresponding $\alpha$-volume of these cylinders is equal to
\begin{align*}
&\sum\limits_{i \in A_k} |\varDelta_{\alpha_1(a_j)\alpha_2(a_j)\ldots\alpha
_{n_j+k}(a_j) ~ i}|^{\alpha}\\
&\quad \leq s \cdot\max\limits_{i \in A_k}
|\varDelta_{\alpha_1(a_j)\alpha_2(a_j)\ldots \alpha_{n_j+k}(a_j)~
i}|^{\alpha}\\
&\quad= s \cdot\max\limits_{i \in A_k} \bigl(|\varDelta_{\alpha_1(a_j)\alpha
_2(a_j)\ldots \alpha_{n_j+k}(a_j)~ i}|^{\alpha-\delta} |\varDelta_{\alpha
_1(a_j)\alpha_2(a_j)\ldots\alpha_{n_j+k}(a_j)~ i}|^{\delta}\bigr)\\
&\quad\leq s |E_j|^{\alpha-\delta} \cdot\max\limits_{i \in A_k} |\varDelta
_{\alpha_1(a_j)\alpha_2(a_j)\ldots\alpha_{n_j+k}(a_j)~ i}|^{\delta}\\
&\quad\leq s |E_j|^{\alpha-\delta} \cdot\left(\prod\limits_{i=2}^{k+1} q_{n_j+i}
~ |\varDelta_{\alpha_1(a_j)\alpha_2(a_j)\ldots\alpha_{n_j}(a_j)}|\right)^{\delta
} \leq s |E_j| ^{\alpha-\delta} \cdot\left(\prod\limits_{i=2}^{k+1}
q_{n_j+i}\right)^\delta.
\end{align*}

So, the set $(a_j, c_j)$ can be covered by a countable family of
cylinders from $L_1$, $L_2, \dots, L_k,\dots$. The total $\alpha
$-volume of all these cylinders does not exceed the value
\[
s |E_j|^{\alpha-\delta} \sum\limits_{k=1}^{\infty} \left(\prod\limits
_{i=2}^{k+1} q_{n_j+i}\right)^\delta\leq S(\delta) \cdot s ~ |E_j|^{\alpha
-\delta}.
\]

Now let us estimate the $\alpha$-volume of the set $[d_j, b_j)$.

Let $R_1=R_1(j)$ be the family of all cylinders of rank $n_j+2$
belonging to the cylinder $\varDelta_{\alpha_1(b_j)\alpha_2(b_j)\ldots
\alpha_{n_j+1}(b_j)}$ and to the set $[d_j, b_j)$.
Let
\[
B_1=B_1(j):= \bigl\{i: i\in\bigl\{ 0, \dots, \alpha_{n_j+2}(b_j)- 1 \bigr\} \bigr\}.\vadjust{\eject}
\]

The corresponding $\alpha$-volume of these cylinders is equal to
\begin{align*}
&\sum\limits_{i \in B_1} |\varDelta_{\alpha_1(b_j)\alpha_2(b_j)\ldots\alpha
_{n_j+1}(b_j) ~ i}|^{\alpha} \\
&\quad\leq s \cdot\max\limits_{i \in B_1}
|\varDelta_{\alpha_1(b_j)\alpha_2(b_j)\ldots \alpha_{n_j+1}(b_j)~
i}|^{\alpha}\\
&\quad= s \cdot\max\limits_{i \in B_1} \bigl(|\varDelta_{\alpha_1(b_j)\alpha
_2(b_j)\ldots \alpha_{n_j+1}(b_j)~ i}|^{\alpha-\delta} |\varDelta_{\alpha
_1(b_j)\alpha_2(b_j)\ldots\alpha_{n_j+1}(b_j)~ i}|^{\delta}\bigr)\\
&\quad\leq s |E_j|^{\alpha-\delta} \cdot\max\limits_{i \in B_1} |\varDelta
_{\alpha_1(b_j)\alpha_2(b_j)\ldots\alpha_{n_j+1}(b_j)~ i}|^{\delta}\\
&\quad\leq s |E_j|^{\alpha-\delta} \cdot\bigl(q_{n_j+2} ~ |\varDelta_{\alpha
_1(b_j)\alpha_2(b_j)\ldots\alpha_{n_j}(b_j)}|\bigr)^{\delta} \leq s \cdot
|E_j| ^{\alpha-\delta} \cdot q_{n_j+2}^{\delta}.
\end{align*}

Similarly, for $k>1$, let $R_k=R_k(j)$ be the family of all cylinders
of rank $n_j+k+1$ belonging to the cylinder $\varDelta_{\alpha_1(b_j)\alpha
_2(b_j)\ldots\alpha_{n_j+k}(b_j)}$ and to the set $[d_j, b_j)$.
Let
\[
B_k=B_k(j):= \bigl\{i: i\in\bigl\{0, \dots, \alpha_{n_j+k+1}(b_j)-1 \bigr\} \bigr\}.
\]

The corresponding $\alpha$-volume of these cylinders is equal to
\begin{align*}
&\sum\limits_{i \in R_k} |\varDelta_{\alpha_1(b_j)\alpha_2(b_j)\ldots\alpha
_{n_j+k}(b_j) ~ i}|^{\alpha} \\
&\quad\leq s \cdot\max\limits_{i \in B_k}
|\varDelta_{\alpha_1(b_j)\alpha_2(b_j)\ldots \alpha_{n_j+k}(b_j)~
i}|^{\alpha}\\
&\quad= s \cdot\max\limits_{i \in B_k} \bigl(|\varDelta_{\alpha_1(b_j)\alpha
_2(b_j)\ldots \alpha_{n_j+k}(b_j)~ i}|^{\alpha-\delta} |\varDelta_{\alpha
_1(b_j)\alpha_2(b_j)\ldots\alpha_{n_j+k}(b_j)~ i}|^{\delta}\bigr)\\
&\quad\leq s |E_j|^{\alpha-\delta} \cdot\max\limits_{i \in B_k} |\varDelta
_{\alpha_1(b_j)\alpha_2(b_j)\ldots\alpha_{n_j+k}(b_j)~ i}|^{\delta}\\
&\quad\leq s |E_j|^{\alpha-\delta} \cdot\left(\prod\limits_{i=2}^{k+1}
q_{n_j+i}\right)^\delta~ |\varDelta_{\alpha_1(b_j)\alpha_2(b_j)\ldots\alpha
_{n_j}(b_j)}|^{\delta} \leq s \cdot|E_j| ^{\alpha-\delta} \left(\prod\limits
_{i=2}^{k+1} q_{n_j+i}\right)^\delta.
\end{align*}

So, the set $[d_j, b_j)$ can be covered by a countable family of
cylinders from $R_1$, $R_2, \dots, R_k,\dots$. The total $\alpha
$-volume of these cylinders does not exceed the value
\[
s |E_j|^{\alpha-\delta} \sum\limits_{k=1}^{\infty} \left(\prod\limits
_{i=2}^{k+1} q_{n_j+i}\right)^\delta\leq S(\delta) \cdot s ~ |E_j|^{\alpha
-\delta}.
\]

Hence, the interval $(a_j, b_j)$ can be covered by using at most $s$
cylinders from $M_0=M_0(j)$, a countable family of cylinders from $L_1,
L_2, \dots, L_k,\dots$ and a countable family of cylinders from $R_1,
R_2, \dots, R_k, \dots$. We emphasize that all these cylinders are
subsets of $(a_j, b_j)$ and their total $\alpha$-volume does not exceed
the value
\[
\bigl(1+ 2S(\delta)\bigr) \cdot s ~ |E_j|^{\alpha-\delta}, \quad\forall\delta\in
(0, \alpha).
\]

Therefore, given a subset $E$, $\alpha\in(0,1]$, $\delta\in(0, \alpha
)$, $\varepsilon>0$, and an $\varepsilon$-covering of the set $E$ by
intervals $(a_j, b_j), ~ a_j\in A, b_j \in A$, there exists an
$\varepsilon$-covering of~$E$ by cylinders from $\varPhi(Q^*)$ such that
its $\alpha$-volume does not exceed the value
\[
\bigl(1+ 2 S(\delta)\bigr) \cdot s ~ \sum\limits_j |E_j|^{\alpha-\delta}.
\]

Hence, for any $\alpha\in(0,1]$, $\delta\in(0, \alpha)$, and
$E\subset[0, 1]$, we have
\[
H^\alpha(E) \leq H^\alpha\bigl(E, \varPhi\bigl(Q^*\bigr) \bigr)\leq\bigl(1+ 2 S(\delta)\bigr) \cdot s
~ H^{\alpha-\delta}(E).
\]

So,
\[
\dim_H\bigl(E,\varPhi\bigl(Q^*\bigr)\bigr) \leq\dim_H(E)+\delta,\quad \forall\delta\in(0, \alpha
),
\]
which proves the inequality
\[
\dim_H\bigl(E, \varPhi\bigl(Q^*\bigr)\bigr) \leq\dim_H(E), \quad \forall E\subset[0, 1].
\]
Therefore,
\[
\dim_H\bigl(E, \varPhi\bigl(Q^*\bigr)\bigr)=\dim_H(E)
\]
for any $E\subset[0, 1]$,
which proves the theorem.
\end{proof}

\begin{corollary}\label{corollary 1}
If
\begin{equation}\label{dostatnya umova dovirchosta <1}
\sup\limits_{ik} q_{ik}<1,
\end{equation}
then the family $\varPhi(Q^*)$ of cylinders generated by $Q^*$-expansion
of real numbers is faithful for the calculation of the Hausdorff
dimension on the unit interval.
\end{corollary}
\begin{proof}
If $\sup_{ik} q_{ik}<1$, then there exists a positive constant
$q<1$ such that $q_{k}< q$ for all $k\in N$.
In such a case,
\[
S(m, \delta) := \sum\limits_{k=1}^{\infty} \left( \prod
_{i=m+1}^{m+k}q_i\right)^{\delta} \leq \sum\limits_{k=1}^{\infty}
q^{k\delta} = \frac{q^{\delta}}{1- q^{\delta}}, \quad \forall m \in N,
\]
so that
\[
S(\delta) := \sup\limits_m S(m, \delta) \leq\frac{q^{\delta}}{1-
q^{\delta}} < +\infty, \quad \forall\delta>0.
\]
Therefore, the family $\varPhi(Q^*)$ is faithful.
\end{proof}

From this corollary it follows, in particular, that the family $\varPhi
(Q^*)$ of cylinders generated by the matrix
\[
Q^*=
\begin{pmatrix}

\frac1{10}&\ldots&\frac1{10^k}&\ldots\\[3pt]

{{\frac12}-{\frac1{10}}}&\ldots&\
{{\frac12}-{\frac1{10^k}}}&\ldots\\[3pt]

{{\frac12}-{\frac1{10}}}&\ldots&\
{{\frac12}-{\frac1{10^k}}}&\ldots\\[3pt]

\frac1{10}&\ldots&\frac1{10^k}&\ldots
\end{pmatrix}
\]
is faithful.

Let us show how sufficient conditions for the faithfulness obtained in
\cite{AlbeverioTorbin2005} can be easily derived from our results.
\begin{corollary}
If $\inf\limits_{k}\{q_{0k}, q_{(s-1)k}\} >0$, then the family $\varPhi
(Q^*)$ of cylinders generated by $Q^*$-expansion of real numbers is
faithful for the calculation of the Hausdorff--Besicovitch dimension on
the unit interval.
\end{corollary}
\begin{proof}
If $\inf_{k}\{q_{0k}, q_{(s-1)k}\} >0$, then there exists a
positive constant $q_*$ such that $q_{0k}>q_*, q_{(s-1)k}>q_*, ~
\forall k\in N$. Therefore,
$ \sup_{ik} q_{ik}\leq1-2q_*<1$. So, the faithfulness of $\varPhi
(Q^*)$ follows from the previous corollary.\vadjust{\eject}
\end{proof}

From the proof of the corollary it follows that it is easy to extend
the results of \cite{AlbeverioTorbin2005} in the following way.

\begin{corollary}
If $\inf_{k}\{q_{0k}\} >0$, then the family $\varPhi(Q^*)$ of
cylinders generated by $Q^*$-expansion of real numbers is faithful for
the calculation of the Hausdorff dimension on the unit interval.
\end{corollary}
\begin{proof}
If $\inf_{k}\{q_{0k}\} >0$, then there exists a positive
constant $q_*$ such that $q_{0k}>q_*, ~ \forall k\in N$. Therefore,
$ \sup_{ik} q_{ik}\leq1-q_*<1$. So, the faithfulness of $\varPhi
(Q^*)$ follows from Corollary \ref{corollary 1}.
\end{proof}
So, for the faithfulness of the family $\varPhi(Q^*)$, it suffices to
control only elements of the first raw of the matrix $Q^*$.

Let us also mention that Theorem \ref{Th2.1} can give a positive
answer on the faithfulness of $\varPhi(Q^*)$ even for the case where $\inf
_{k}\{q_{0k}, q_{(s-1)k}\} =0$ and $\sup_{ik} q_{ik}=1$
simultaneously. To illustrate this, let us consider the matrix

\[
Q^*=
\begin{pmatrix}

\frac1{4}&\frac{1}{3}&\ldots&\frac{1}{2^{n+1}}&\frac{1}{3}&\ldots\\[3pt]

\frac1{2}&\frac{1}{3}&\ldots&\frac{2^n-1}{2^{n}}&\frac{1}{3}&\ldots\\[3pt]

\frac1{4}&\frac{1}{3}&\ldots&\frac{1}{2^{n+1}}&\frac{1}{3}&\ldots\\[3pt]
\end{pmatrix}
,
\]
that is,
\[
q_{0k}=q_{1k}=q_{2k}=\frac13, \quad k=2n, ~ n\in N,
\]
and
\[
q_{0k}=q_{2k}=\frac{1}{2^{n+1}},\qquad q_{1k}= \frac{2^n-1}{2^n}, \quad k=2n-1, ~
n\in N.
\]

In such a case, $\inf_{k}\{q_{0k}, q_{(s-1)k}\} =0$ and $\sup
_{ik} q_{ik}=1$, but it is clear that $q_k q_{k+1}< \frac13$ for
all $k\in N$, and, therefore,
\[
S(m, \delta) := \sum\limits_{k=1}^{\infty} \left( \prod\limits
_{i=m+1}^{m+k}q_i\right)^{\delta} \leq q_{m+1}^{\delta}+ 2 \sum\limits
_{n=1}^{\infty} \biggl(\frac13\biggr)^{n\delta} = 1+ \frac{2}{3^\delta-1}, \quad\forall
m \in N.
\]
%

So, $ S(\delta) < +\infty$ for all $\delta>0$, which proves the
faithfulness of the family~$\varPhi(Q^*)$.

\section{On fine fractal properties of random variables with
independent $Q^{*}$-symbols}

Let $\{\xi_k\}$ be a sequence of independent random variables taking
values $0, 1, \dots$, $s-1$ with probabilities $p_{0k}, p_{1k}, \dots,
p_{s-1 k}$, respectively. The random variable
\begin{equation}
\xi= \varDelta^{Q^*}_{\xi_1 \xi_2 \dots\xi_k\dots}
\end{equation}
is said to be the random variable with independent $Q^*$-symbols. Let
$\nu_\xi$ be the corresponding probability measure.

The Lebesgue structure of $\nu_\xi$ is well studied (see, e.g., \cite
{AlbeverioTorbin2005, AKPT2011}). It is known, in particular, that the
distribution of $\xi$ is of pure type. It is of pure discrete type if\vadjust{\eject}
and only if
\begin{equation}\label{DiscrCond}
\prod_{k=1}^\infty\max_i
p_{ik}>0,
\end{equation}
of pure absolutely continuous type if and only if
\begin{equation}\label{AbsCont}
\prod_{k=1}^\infty\left( \sqrt{p_{0k}q_{0k}}+
\sqrt{p_{1k}q_{1k}}+\cdots+ \sqrt{p_{(s-1)k}q_{(s-1)k}}\right)>0,
\end{equation}
and of pure singularly continuous type if and only if infinite products
(\ref{DiscrCond}) and~(\ref{AbsCont}) are equal to zero.

Let us recall that the Hausdorff dimension of the distribution of a
random variable $\tau$ is defined as follows:
\[
\dim_{H} (\tau)=\mathop{\inf} \bigl\{\dim_H (E),\, \, E\in\mathcal
{B_{\tau}}\bigr\},
\]
where $\mathcal{B}_{\tau}$ is the family of all possible (not
necessarily closed) supports of the random variable $\tau$, that is,
\[
\mathcal{B}_{\tau} =\bigr\{E:\, E\in\mathcal{B},\, \, P_{\tau}
(E)=1\bigl\}.
\]

Let us also recall the notion of the Hausdorff--Billingsley dimension
of a~set w.r.t.\ a probability measure and w.r.t.\ a system of partitions.
Let $\upsilon$ be a~continuous probability measure on $[0,1]$, and let
$\varPhi$ be the family of cylinders generated by some expansion. Then the
$(\upsilon\, -\, \alpha)$-Hausdorff measure of a set $E \subset
[0,1]$ w.r.t.\ the family $\varPhi$ and measure $\upsilon$ is defined as follows:
\[
H^{\alpha} (E,\, \upsilon,\varPhi)=\mathop{\lim}\limits_{\varepsilon
\to0 } \biggl[\mathop{\mathop{\inf}\limits_{\upsilon(E_{j} )\le
\varepsilon} \biggl\{\sum_{j}\upsilon^{\alpha} (E_{j} ) \biggr\}}
\biggr]=\mathop{\lim}\limits_{\varepsilon\to0 } H_{\varepsilon
}^{\alpha} (E,\, \upsilon,\varPhi),
\]
where $E_{j} \in\varPhi,\, \, \bigcup_{j}E_{j} \supset E$.

The number
\[
\dim_{\upsilon} (E,\, \varPhi)=\inf\bigl\{ \alpha:\, H^{\alpha} (E,\,
\upsilon,\varPhi)\, =0\bigr\}
\]
is called \textit{the Hausdorff--Billingsley dimension} of a set $E$
w.r.t.\  $\upsilon$ and $\varPhi$.

Let
\begin{align*}
h_j &:= - \sum\limits_{i=0}^{s-1} p_{ij} \ln p_{ij}, \qquad 0 \ln0 := 0 , \qquad
H_n := \sum\limits_{j=1}^n h_j,\\
b_j &:= -\sum_{i=0}^{s-1} p_{ij}
\ln q_{ij}, \qquad B_n := \sum_{j=1}^n b_j,
\end{align*}

and
\[
d_j := - b^2_j + \sum_{i=0}^{s-1} p_{ij}
\ln^2 q_{ij}.\vadjust{\eject}
\]

\begin{theorem} Assume that the following conditions hold:
\begin{align}\label{sufficient condition for Q*}
&S(\delta) < +\infty, \quad \forall\delta>0;\\
\label{umova na dyspersiyu}
&\sum\limits_{j=1}^{\infty} \frac{d_j}{j^2} < +\infty;\\
\label{umova na vidokremlenist vid 0 B_n}
&\varliminf\limits_{{ n \to\infty}} \frac{B_n}{n} > 0.
\end{align}
Then the Hausdorff dimension of the distribution of a random variable
with independent $Q^*$-symbols is equal to
\begin{equation}\label{dim_H nu_xi}
\dim_H \nu_\xi= \varliminf\limits_{{ n \to
\infty}}\frac{H_n}{B_n}.
\end{equation}
\end{theorem}

\begin{proof}
Let $\varDelta_n(x)= \varDelta^{Q^*}_{\alpha_1(x) \alpha_2(x)\dots
\alpha_n(x) }$ be the cylinder of rank $n$ of the $Q^*$-expansion of
$x$. Let $\nu=\nu_\xi$, and let
$\mu$ be the Lebesgue measure on $[0,1]$.

Then
\begin{align*}
\nu\bigl(\varDelta_n(x)\bigr)&= p_{\alpha_1(x)1} \cdot
p_{\alpha_2(x)2} \cdot\ldots\cdot p_{\alpha_n(x)n} ,\\
\mu
\bigl(\varDelta_n(x)\bigr)&= q_{\alpha_1(x)1} \cdot q_{\alpha_2(x)2} \cdot\ldots
\cdot q_{\alpha_n(x)n} .
\end{align*}

Let us consider
\[
\frac{\ln\nu
(\varDelta_n(x))}{ \ln\mu(\varDelta_n(x))}= \frac{\sum_{j=1}^n
\ln p_{\alpha_j(x)j}}{\sum_{j=1}^n \ln q_{\alpha_j(x)j}}.
\]

If a real number $x= \varDelta^{Q^*}_{\alpha_1(x) \alpha_2(x)\dots\alpha
_n(x)\dots}$
is chosen randomly so that\break $P(\alpha_j(x)=i)= p_{ij}$
(i.e., the distribution of the random variable $x$ coincides with the
initial probability measure $\nu$),
then $\{ \eta_j \}= \{ \eta_j(x)\} = \{ \ln p_{\alpha_j(x)j}\}$
and $\{ \psi_j \}= \{ \psi_j(x)\} = \{ \ln q_{\alpha_j(x)j}\}$ are
sequences of independent random variables with the following distributions:

\begin{table}[h!]
\begin{tabular}{|c|c|c|c|c|}
\hline
$\eta_j$ & $\ln p_{0j}$ & $\ln p_{1j}$ & \ldots & $\ln p_{(s-1) j}$ \\[2pt]
\hline
$~$ & $p_{0j}$ & $p_{1j}$ & \ldots & $p_{(s-1)j}$ \\[2pt] \hline
\end{tabular}
$~~~~$
\begin{tabular}{|c|c|c|c|c|}
\hline
$\psi_j$ & $\ln q_{0j}$ & $\ln q_{1j}$ & \ldots & $\ln q_{(s-1) j}$ \\[2pt]
\hline
$~$ & $p_{0j}$ & $p_{1j}$ & \ldots & $p_{(s-1)j}$ \\[2pt] \hline
\end{tabular}
\end{table}

It is clear that
$ M \eta_j = -h_j$ and $ |h_j| \le\ln s$.

It is not hard to check that $ M \eta_j^2 = \sum_{i=0}^{s-1} p_{ij}\ln
^2 p_{ij} \leq\frac{4}{e^2}$ \cite{AlbeverioTorbin2005}.

From the strong law of large numbers \cite{Sh} it follows that, for $\nu
$-almost all $x \in[0,1]$, the following equality holds:
\begin{equation}\label{Kol}
\lim_{n \to\infty}
\frac{(\eta_1+\eta_2+\cdots+\eta_n)- M(\eta_1+\eta_2+\cdots+
\eta_n)}{n}=0.
\end{equation}
It is clear that $M(\eta_1+\eta_2+\cdots+\eta_n)= -H_n$.

To show that the strong law of large numbers can also be applied to the
sequence $\{ \psi_j \}$, let us consider
\[
M \psi_j=
\sum_{i=0}^{s-1} p_{ij} \ln q_{ij}, \qquad M \psi^2_j=
\sum_{i=0}^{s-1} p_{ij} \ln^2 q_{ij}.
\]

Since $d_j = D (\psi_j)$ and the series $\sum_{j=1}^{\infty}
\frac{d_j}{j^2}$ converges (see (\ref{umova na dyspersiyu})), by
Kolmogorov's theorem (strong law of large numbers \cite{Sh}) it follows
that, for $\nu$-almost all $x \in[0,1]$,
\begin{equation}\label{Kol2}
\lim_{n \to\infty}
\frac{(\psi_1+\psi_2+\cdots+\psi_n)- M(\psi_1+\psi_2+\cdots+
\psi_n)}{n}=0.
\end{equation}

Let us remark that $M(\psi_1+\psi_2+\cdots
+\psi_n)= -B_n$.

Now let us consider the set
\begin{align*}
A&= \bigg\{ x:  \lim_{n \to\infty}
\bigg( \frac{\eta_1(x)+ \eta_2(x)+\cdots+ \eta_n(x)}{\psi_1(x)+
\psi_2(x)+\cdots+ \psi_n(x)} - \frac{H_n}{B_n} \bigg) =0 \bigg\}\\
&= \bigg\{ x:  \lim_{n \to\infty} \frac{
( \frac{ \eta_1(x)+ \eta_2(x)+\cdots+ \eta_n(x)+H_n}{n} )-
\frac{H_n}{B_n} ( \frac{ \psi_1(x)+ \psi_2(x)+\cdots+
\psi_n(x)+B_n}{n} ) }{( \frac{ \psi_1(x)+
\psi_2(x)+\cdots+ \psi_n(x)+ B_n}{n} ) - \frac{B_n}{n}} =0
\bigg\}.
\end{align*}

By the Gibbs inequality it follows that $h_j\leq b_j$. Hence,
$0\leq \frac{H_n}{B_n}=\frac{\sum_{j=1}^n h_j}{\sum_{j=1}^n b_j}\leq1$.

Since $\varliminf_{{ n \to\infty}} \frac{B_n}{n} > 0$ (see
(\ref{umova na vidokremlenist vid 0 B_n})), we deduce the existence of
a constant $c_1 >0$ such that $ | \frac{B_n}{n} | \ge c_1$
for all $n \in N$.

Therefore, for $\nu$-almost all $x \in[0,1] $,
\[
\lim_{n \to\infty} \frac{ (
\frac{ \eta_1(x)+ \eta_2(x)+\cdots+ \eta_n(x)+H_n}{n} )-
\frac{H_n}{B_n} ( \frac{ \psi_1(x)+ \psi_2(x)+\cdots+
\psi_n(x)+B_n}{n} ) }{( \frac{ \psi_1(x)+ \psi_2(x)+\cdots+
\psi_n(x)+B_n}{n} ) - \frac{B_n}{n}} =0.
\]
So, $\nu(A)=1$ and $\dim_\nu(A, \varPhi)=1$.

Let us consider the sets
\begin{align*}
A_1&= \bigg\{ x: ~~
\varliminf_{{n \to\infty}} \bigg( \frac{\eta_1(x)+
\eta_2(x)+\cdots+ \eta_n(x)}{\psi_1(x)+ \psi_2(x)+\cdots+ \psi_n(x)} -
\frac{H_n}{B_n} \bigg) =0 \bigg\},\\
A_2&= \bigg\{ x: ~~ \varliminf_{{n \to\infty}} \bigg(
\frac{\eta_1(x)+ \eta_2(x)+\cdots+ \eta_n(x)}{\psi_1(x)+
\psi_2(x)+\cdots+ \psi_n(x)} \bigg) \leq
\varliminf_{{n \to\infty}} \frac{H_n}{B_n} \bigg\} \\
&= \bigg\{ x: ~~
\varliminf_{{n \to\infty}} \frac{ \ln\nu
(\varDelta_n(x))}{\ln\mu(\varDelta_n(x))} \leq
\varliminf_{{n \to\infty}} \frac{H_n}{B_n} \bigg\},
\end{align*}

and
\begin{align*}
A_3&= \bigg\{ x: ~~ \varliminf_{{n \to\infty}} \bigg(
\frac{\eta_1(x)+ \eta_2(x)+\cdots+ \eta_n(x)}{\psi_1(x)+
\psi_2(x)+\cdots+ \psi_n(x)} \bigg) \geq
\varliminf_{{n \to\infty}} \frac{H_n}{B_n} \bigg\} \\
&= \bigg\{ x: ~~
\varliminf_{{n \to\infty}} \frac{ \ln\nu
(\varDelta_n(x))}{\ln\mu(\varDelta_n(x))} \geq
\varliminf_{{n \to\infty}} \frac{H_n}{B_n} \bigg\}.
\end{align*}

It is obvious that $A \subset A_1$. By the same arguments as in \cite
{AlbeverioTorbin2005} we can easily check that
$A_1 \subset A_3$ and $A \subset A_2$.

Let $D= \varliminf_{{n \to\infty}}
\frac{H_n}{B_n}$.

From $A\subset A_2$ it follows that $ \dim_\mu(A,
\varPhi) \leq \dim_\mu(A_2, \varPhi)$. From Theorem~2.1 of \cite
{Billingsley61} it follows that
$ \dim_\mu(A_2, \varPhi) \le D$. Therefore,
$\dim_\mu(A, \varPhi) \le D$.

From the condition $ A \subset A_3$ and Theorem 2.2 of \cite
{Billingsley61} it follows that $ \dim_\mu
(A, \varPhi) \ge D \cdot\dim_\nu(A, \varPhi)= D \cdot1
= D$. Hence, $\dim_\mu(A, \varPhi)= D$.

Since $\mu$ is the Lebesgue measure on $[0,1]$, we get $\dim_H (A,
\varPhi)=\break
\dim_\mu(A, \varPhi)=D$. From (\ref{sufficient condition for Q*}) and
Theorem \ref{teorema 2 o ddoveritelnosti} it follows that the family
$\varPhi$ of cylinders of the $Q^*$-expansion is faithful for the
determination of the Hausdorff--Besicovitch dimension on the unit
interval, and, therefore,
$\dim_H(A, \varPhi)= \dim_H (A)$. Hence, $\dim_H(A)=D$.

Finally, let us prove that the constructed set $A$ is the minimal
dimensional support of the measure $\nu$. To this end, let us consider
an arbitrary support $C$ of the measure $\nu$. It is clear that the set
$C_1= C \cap A$ is also a support of the measure $\nu$ and that $C_1
\subset C$.
Then $\dim_H(C_1) \le\dim_H(C)$ and $C_1 \subset A$.

Let us prove that $\dim_H(C_1)= \dim_H(A)$.

From $C_1 \subset A$ it follows that $\dim_H(C_1) \le
\dim_H(A) = D$. On the other hand,
\[
C_1 \subset A
\subset A_3= \bigg\{ x: ~~ \varliminf_{{n \to\infty}}
\frac{ \ln\nu(\varDelta_n(x))}{\ln\mu(\varDelta_n(x))} \ge D
\bigg\}.
\]
Therefore, from Theorem 2.2 of \cite{Billingsley61} it follows that
\[
\dim_H(C_1) = \dim_\mu(C_1, \varPhi) \ge D \cdot
\dim_\nu(C_1, \varPhi) = D \cdot1 = D.
\]
So, $\dim_H(C_1)= D =\dim_H(A)$.
\end{proof}

\section*{Acknowledgments}
This work was partly supported by research projects ``Spectral
Structures and Topological Methods in Mathematics'' (SFB-701, Bielefeld
University), STREVCOM\break FP-7-IRSES 612669 (EU), ``Multilevel analysis of
singularly continuous probability measures and its applications''
(Ministry of Education and Science of Ukraine), and by Alexander von
Humboldt Stiftung.


%
\end{document}